\documentclass[12pt]{article}

\usepackage{fullpage}

\usepackage[backref,colorlinks,citecolor=blue,bookmarks=true]{hyperref}

\usepackage{amsmath, amsthm, amssymb}

\makeatletter
\renewenvironment{proof}[1][\proofname]
{\par\pushQED{\qed}
	\normalfont\topsep6\p@\@plus6\p@\relax\trivlist
	\item[\hskip\labelsep\bfseries#1\@addpunct{.}]
	\ignorespaces}
{\popQED\endtrivlist\@endpefalse}
\makeatother

\newtheorem{theo}{Theorem}
\newtheorem{observation}{Observation}
\newtheorem{prop}{Theorem}[section]
\newtheorem{lemma}[prop]{Lemma}

\newtheorem{claim}[prop]{Claim}

\newtheorem*{remark*}{Remark}
\newtheorem{fact}[prop]{Fact}

\newtheorem*{notation}{Notation}
\newtheorem{definition}{Definition}

\newcommand{\NN}{\mathbb{N}}

\newcommand{\RR}{{\mathbb R}}

\renewcommand{\a}{\alpha}

\renewcommand{\d}{\delta}
\newcommand{\e}{\epsilon}

\newcommand{\D}{\Delta}

\newcommand{\sub}{\subseteq}

\DeclareMathOperator{\poly}{poly}
\DeclareMathOperator{\twr}{twr}

\newcommand{\R}{{\mathcal R}}
\newcommand{\Q}{{\mathcal Q}}

\renewcommand{\P}{{\mathcal P}}
\newcommand{\V}{{\mathcal V}}

\DeclareMathOperator{\sbu}{\odot}
\newcommand{\coeff}[2]{c_{{#1},{#2}}}

\date{}

\title{On Generalized Regularity}
\author{
	Noga Alon\thanks{Department of Mathematics, 
	Princeton University, Princeton,
	NJ 08544, USA and Schools of Mathematics and Computer Science,
	Tel Aviv University, Tel Aviv, Israel. 
	Email: \texttt{nalon@math.princeton.edu}.
	Research supported in part by
	NSF grant DMS-1855464, 
	ISF grant 281/17,
	BSF grant 2018267
	and the Simons Foundation.}
	\and
	Guy Moshkovitz\thanks{Institute for Advanced Study, Princeton, NJ 08540, USA and
	Rutgers University, New Brunswick, NJ 08854, USA.
	Email: \texttt{guymoshkov@gmail.com}.
	This work was conducted at DIMACS and partially enabled through support from the National Science Foundation under grant number CCF-1445755.}
}
\begin{document}
\maketitle

\begin{abstract}
	Szemer{\'e}di's regularity lemma is one instance in a family of 
regularity lemmas, 
	replacing the definition of density of a graph by a 
more general coefficient.
	Recently, Fan Chung proved another instance, a regularity lemma for 
	clustering graphs, 
	and asked whether good upper bounds could be derived for the 
quantitative estimates it supplies. 
	We answer this question in the negative, for every generalized regularity lemma.
\end{abstract}

\section{Introduction}

Szemer{\'e}di's regularity lemma~\cite{Szemeredi78} is a cornerstone of extremal combinatorics, with applications in graph theory, number theory, computer science and more.
Underlying the lemma is the notion of \emph{density} of a graph, which is the number of edges divided by the total number of vertex pairs.
The lemma says, roughly, that any graph of density bounded away from zero 
can be approximated by a union of a constant number of bipartite graphs $G_i$ that are regular, meaning that all induced subgraphs of a $G_i$ with sufficiently many vertices have approximately the same density.

Recently, Chung~\cite{Chung19} proved a variant of Szemer{\'e}di's 
regularity lemma that is tailor made for \emph{clustering graphs}, 
which are graph with a \emph{clustering coefficient} bounded away from zero. 
The clustering coefficient of a graph $G$ is the number of 
triangles in $G$ divided by the number of two-edge paths in $G$.
It plays a key role in the definition of the ``small world phenomenon'',  
as clustering graphs come up often in real-world settings, such as social and neural networks~(\cite{LucePe49,WattsSt98}).
Chung's regularity lemma says, roughly, that every graph $G$ with a 
clustering coefficient bounded away from zero 
can be approximated by a union of a constant number of induced tripartite subgraphs $G_i$, such that all induced subgraphs of a $G_i$ with sufficiently many two-edge paths have approximately the same clustering coefficient.
Importantly, this regularity lemma is meaningful even if the graph $G$
does not have constant density, unlike Szemer{\'e}di's regularity lemma, as long as the clustering coefficient of $G$ is constant.

More generally, for any fixed graph $H$ and subgraph $F$ of $H$ on the same vertex set, one may consider the analogous \emph{$(H,F)$-coefficient}, namely, the number of copies of $H$ divided by the number of copies of $F$. 
It turns out that an ``$(H,F)$-regularity lemma'' holds in general (see~\cite{Chung19}).

Since clustering graphs often appear ``in nature'', 
it is perhaps reasonable to suspect that they have an efficient 
regularity lemma, meaning one where the number of parts in the partition 
is a slow-growing function of the regularity parameter $1/\e$. 
Indeed, the existence of an efficient regularity lemma was the first open question raised in Chung's paper~\cite{Chung19}, which adds that ``It is of both theoretical
and practical interest to see if the clustering property could be helpful for
reducing the size of the partition''.
More generally, one might ask whether there is an efficient $(H,F)$-regularity lemma for \emph{some} pairs $(H,F)$.
Here we prove that the answer to all these questions is negative.

\subsection{Generalized regularity}

Denote by $n_H(G)$, for a $k$-vertex graph $H$ and a $k$-partite graph $G$, 
the number of 
unlabeled 
copies\footnote{That is, subgraphs of $G$ that are isomorphic to $H$. Alternatively, we may consider labeled copies (injective mappings $V(H) \to V(G)$ that map edges of $H$ to edges $G$) with no real changes in our results.}
of $H$ in $G$ 
where distinct vertices are mapped into distinct vertex classes.
%
For a subgraph $F \sub H$, also on $k$ vertices, the $(H,F)$-coefficient of a $k$-partite graph $G$ is 
$$\coeff{H}{F}(G) = \frac{n_H(G)}{n_F(G)} .$$
Note that the density $d(G)$ of a bipartite graph $G$ is a special case of a $\coeff{H}{F}(G)$; specifically,
if the vertex classes of $G$ are $(V_1,V_2)$ then $\coeff{K_2}{\overline{K_2}}(G) = |G|/|V_1||V_2| = d(G)$.
More generally, $\coeff{H}{\overline{K_k}}(G)$ is the ``$H$-density'' (that is, normalized $H$ count) of $G$.
Moreover, $\coeff{K_3}{P_2}(G)$ is the clustering coefficient of $G$, where $P_2$ denotes the $2$-edge path graph.\footnote{We note that~\cite{Chung19} has, for technical reasons, two definitions of the clustering coefficient, depending on whether the graph is guaranteed to be tripartite or not, and in the former case the denominator is the number of two-edge paths with some fixed orientation. 
	We opted to use a single definition.}

Let us formally define $(H,F)$-regularity. 
We use the notation $x \pm \e$ for a number lying in the interval $[x-\e,\,x+\e]$.

\begin{definition}[$(H,F)$-regular graph]
	Let $F \sub H$ be graphs on $k$ vertices.
	A $k$-partite graph $G$ is \emph{$\e$-$(H,F)$-regular} if for every induced $k$-partite subgraph $G'$ with $n_F(G') \ge \e n_F(G)$, we have $\coeff{H}{F}(G') = \coeff{H}{F}(G) \pm \e$.
\end{definition}

\begin{definition}[$(H,F)$-regular partition]
	Let $F \sub H$ be graphs on $k$ vertices.
	A vertex partition $\P$ of a graph $G$ is \emph{$\e$-$(H,F)$-regular} if all but at most $\e n_H(G)$ copies of $H$ in $G$ lie in induced $k$-partite subgraphs $G[X_1,\ldots,X_k]$, with $X_1,\ldots,X_k \in \P$, that are $\e$-$(H,F)$-regular.
\end{definition}

Our result shows that in any proof of an $(H,F)$-regularity lemma, the order of the partition (that is, its number of parts) can in general be as large as a tower function, 
similarly to the case for Szemer{\'e}di's regularity lemma---as originally shown by Gowers~\cite{Gowers97}.
In fact, the graphs we construct use Gowers' construction (or any other construction witnessing tower-type bounds) as a black box.
Our proof proceeds by showing that for certain special graphs, as defined in Section~\ref{sec:proof}, having an $(H,F)$-regular partition can be used to derive a similarly regular partition.

We are now ready to state our result.
Formally, the tower function is defined recursively by $\twr(n)=2^{\twr(n-1)}$ for $n \ge 1$ and $\twr(0)=1$.
To simplify the presentation, we omit all floor and ceiling signs whenever these are not crucial. 

\begin{theo}\label{theo:main}
	Let $F \subsetneq H$ be graphs on $k$ vertices.
	There is a graph whose every $\e$-$(H,F)$-regular partition is of order at least $\twr(1/\poly(\e k^{k}))$.
\end{theo}


\section{Lower Bound Proof}\label{sec:proof}

\subsection{Semi-blowups}

The central definition in our proof is that of a \emph{semi-blowup}.
A \emph{blowup} of a graph $H$ is any graph obtained by replacing every vertex $i$ of $H$ by a (non-empty) independent set $V_i$ of new vertices, each edge $\{i,j\}$ of $H$ by a complete bipartite graph between $V_i$ and $V_j$, and each non-edge $\{i,j\}$ of $H$ by an empty bipartite graph between $V_i$ and $V_j$.
A \emph{semi-blowup of $H$} is any graph obtained from a blowup of $H$ by replacing the bipartite graph between $V_i$ and $V_j$---for only one choice of $\{i,j\}$---by any bipartite graph.
We write $H \sbu_e G_0$ for a semi-blowup of $H$ in which 
$G_0$ replaces the bipartite graph corresponding to $e=\{i,j\}$. 

We write $\NN=\{0,1,2,\ldots\}$ for the set of nonnegative integers, and $[k]=\{1,\ldots,k\}$.

\begin{notation}
	In everything that follows, we fix:
	\begin{center}
	an integer $k \ge 2$, graphs $F \subsetneq H$ on the vertex set $[k]$,\\
	and an edge $e=\{1,2\} \in E(H) \setminus E(F)$.
	\end{center}
\end{notation}
Henceforth, in the semi-blowup $H \sbu_e G_0$, the first vertex class of $G_0$ is embedded as $V_1$ and the second as $V_2$.
We will repeatedly use the following easy properties of semi-blowups.
The first property shows that an induced subgraph of a semi-blowup of $H$ is again a semi-blowup of $H$.
\begin{observation}\label{obs:sbu-induced}
	For all $U_1 \sub V_1,\ldots,U_k \sub V_k$, 
	the induced subgraph $(H \sbu_e G_0)[U_1,\ldots,U_k]$ is of the form $H \sbu_e G_0[U_1,U_2]$.
\end{observation}

The second property shows that the ``$H$-density'' of a semi-blowup of $H$ is exactly the (edge-) density of the replacement bipartite graph.
\begin{observation}\label{obs:sbu-copies}
	$n_H(H \sbu_e G_0)=d(G_0)|V_1|\cdots|V_k|$.
\end{observation}
\begin{proof}
	Since $e = \{1,2\} \in E(H)$, 
	a choice of a vertex from each vertex class, $v_1\in V_1,\ldots,v_k \in V_k$, spans a copy of $H$ if and only if $\{v_1,v_2\} \in E(G_0)$. Moreover, such a copy is necessarily an induced copy of $H$, and thus the only copy of $H$ on this set of vertices. We conclude $n_H(H \sbu_e G_0)=|G_0|\cdot|V_3|\cdots|V_k|=d(G_0)|V_1|\cdots|V_k|$.
\end{proof}

Our choice of $F$ and $e$ has the following implication, which will be central to our proofs.

\begin{lemma}\label{lemma:sbu}
	There are integers $a,b \in \NN$, with $a \ge 1$ and $a+b \le k!$, such that for every bipartite graph $G_0$ of density $d$, 
	\begin{equation}\label{eq:sbu-F}
	n_F(H \sbu_e G_0) = (a+bd)|V_1|\cdots|V_k| .
	\end{equation}
	This has the following corollaries:
	\begin{equation}\label{eq:sbu-F-lb}
	n_F(H \sbu_e G_0) \ge |V_1|\cdots|V_k| ,
	\end{equation}
	and, by Observation~\ref{obs:sbu-copies},
	\begin{equation}\label{eq:sbu-coeff}
	\coeff{H}{F}(H \sbu_e G_0) = \frac{d}{a+bd} .
	\end{equation}
\end{lemma}
\begin{proof}
	Put $H^- = (V(H),\, E(H)\setminus\{e\})$.
	For every choice of vertices $v_1 \in V_1,\ldots,v_k \in V_k$, 
the copy of $H$ they induce in $G$ contains $n_F(H)$ copies of $F$ if $\{v_1,v_2\} \in G_0$, and otherwise only $n_F(H^-)$ copies of $F$.
	It follows that
	$$n_F(G) = n_F(H^-)|V_1|\cdots|V_k| + (n_F(H)-n_F(H^-))|G_0||V_3|\cdots|V_k| .$$
	Put $a=n_F(H^-)$ and $b=n_F(H)-a$, so that $n_F(G) = (a+bd)|V_1|\cdots|V_k|$.
	We clearly have $a+b = n_F(H) \le k!$.
	Moreover, and crucially, we have $a \ge 1$, since $F$ has at least one copy in $H$ that does not contain $e$, namely, $F$---as $e \notin E(F)$.
	This completes the proof.
\end{proof}

\subsection{Regularity of semi-blowups}

We use Lemma~\ref{lemma:sbu} to show that a semi-blowup can be used to reduce $(H,F)$-regularity to Szemer{\'e}di's regularity.
Recall that a bipartite graph $G$ on $(V_1,V_2)$ is said to be \emph{$\e$-regular} if for all induced bipartite subgraphs $G[S_1,S_2]$ with $|S_1| \ge \e|V_1|$ and $|S_2| \ge \e|V_2|$ we have $d(G[S_1,S_2]) = d(G) \pm \e$.

\begin{claim}\label{claim:reg-partite}
	For every bipartite graph $G_0$, 
	if $H \sbu_e G_0$ is $\e$-$(H,F)$-regular then $G_0$ is $\sqrt{\e} \cdot k^{2k}$-regular.
\end{claim}
\begin{proof}
	Put $G = H \sbu_e G_0$ and $\e'=\sqrt{\e} \cdot k^{2k}$.
	Let $S_1 \sub V_1$ and $S_2 \sub V_2$ be subsets with $|S_1| \ge \e'|V_1|$ and $|S_2| \ge \e'|V_2|$, and put $d^*=d(G[S_1,S_2])$, $d=d(G[V_1,V_2])$.
	We will prove that  $d^* = d \pm \e'$, which would complete the proof as it would imply that $G_0$ is $\e'$-regular.
	
	By Observation~\ref{obs:sbu-induced}, the induced $k$-partite graph $G[S_1,S_2,V_3\ldots,V_k]$ is a semi-blowup of $H$ of the form $H \sbu_e G[S_1,S_2]$.
	It follows from~(\ref{eq:sbu-coeff}) that there are integers $a,b \in \NN$, 
	with $a \ge 1$ and $a+b \le k!$, 
	such that 
	\begin{equation}\label{eq:coeff-f}
	\coeff{H}{F}(G)=f(d) \quad\text{ and }\quad \coeff{H}{F}(G[S_1,S_2,V_3,\ldots,V_k])=f(d^*)
	\end{equation}
	where
	$$f(x) = \frac{x}{a+bx}\;, \qquad f \colon[0,1] \to \RR .$$
	Note that $f'(x) = a/(a+bx)^2 \ge 0$, so $f$ is monotone increasing and $f(x) \le f(1) = 1/(a+b)$.
	Let 
	$$g(x) = \frac{ax}{1-bx}\;, \qquad g \colon[0,1/b) \to \RR .$$
	Observe that $g$ and $g \circ f$ are well defined and, for every $x \in [0,1]$, we have $g(f(x)) = x$.
	Note that $g'(x) = a/(1-bx)^2$ for every $x \in [0,1/b)$, and so is monotone increasing; this implies the bound
	$$g'\big(f(x)+\e\big) \le  g'\Big(\frac{b+ a/2}{b(a+b)}\Big) = a\Big(\frac{a+b}{a/2}\Big)^2 
	\le k^{2k} ,$$
	where the first inequality assumes $\e \le \frac{1}{k^{2k}}$ (otherwise there is nothing to prove),  
	so $\e \le \frac{a/2}{b(a+b)}$.
	
	Using~(\ref{eq:sbu-F-lb}) we deduce the bound
	$$n_F(G[S_1,S_2,V_3,\ldots,V_k]) \ge |S_1||S_2||V_3|\cdots|V_k| \ge {\e'}^2|V_1|\cdots|V_k| \ge ({\e'}^2/k!) n_F(G) \ge \e \cdot n_F(G) .$$
	Thus, since $G$ is $\e$-$H/F$-regular, $\coeff{H}{F}(G[S_1,S_2,V_3\ldots,V_k]) = \coeff{H}{F}(G) \pm \e$, that is, 
	$f(d^*) = f(d) \pm \e$ by~(\ref{eq:coeff-f}).
	We deduce, using the properties of $f$ and $g$ above, that
	$$d^* = g(f(d^*)) = g(f(d)) \pm \e \cdot \max_{x = f(d) \pm \e} g'(x) = d \pm \e \cdot k^{2k} ,$$
	where the second step applies the mean value theorem on $g$ in the interval between $f(d^*)$ and $f(d)$. 
	Since $\e \cdot k^{2k} \le \e'$, this completes the proof.
\end{proof}


\subsection{Regular partitions---auxiliary claims}


The following claim gives an approximate restriction of a partition onto a subset. 
\begin{claim}[Approximate restriction]\label{lemma:approx-cover}
	Let $\P$ be a partition of $U$ and let $V \sub U$ be a subset. For $\P' = \big\{ X \in \P \,\big\vert\, |X \cap V| \ge \d|X| \big\}$ with $\d=\a|V|/|U|$ we have
	$$\sum_{X \in \P'} |X \cap V| \ge (1-\a)|V|.$$
\end{claim}
\begin{proof}
	We have
	$$\sum_{X \in \P \setminus \P'} |X \cap V| \le \sum_{X \in \P \setminus \P'} \d|X| 
	\le \sum_{X \in \P} \d|X| = \d |U| = \a|V| .$$
	Therefore,
	$$\sum_{X \in \P'} |X \cap V|
	= \sum_{X \in \P} |X \cap V| - \sum_{X \in \P \setminus \P'} |X \cap V|
	\ge (1-\a)|V| .$$
\end{proof}

The \emph{slicing lemma} for regular graphs shows that a sufficiently large induced subgraph of a regular bipartite graph is regular.
We will need the following analogue for $(H,F)$-regularity.

\begin{claim}[Slicing]\label{claim:slicing}
	Let $T$ be a $k$-partite graph on $(X_1,\ldots,X_k)$, 
	and let $Y_1 \sub X_1, \ldots, Y_k \sub X_k$ with $|Y_i| \ge \d_i|X_i|$
	be such that $G:=T[Y_1,\ldots,Y_k]$ is a semi-blowup of $H$.
	If $T$ is $\e$-$(H,F)$-regular then $G$ is $\e'$-$(H,F)$-regular with $\e'=\e \cdot k!/(\d_1\cdots\d_k)$.
\end{claim}
\begin{proof}
	Put $\D=\d_1\cdots\d_k$.
	Since, by assumption, $G$ is a semi-blowup of $H$, we have, using~(\ref{eq:sbu-F-lb}),
	$$n_F(G) \ge |Y_1|\cdots|Y_k| \ge \D|X_1|\cdots|X_k| \ge \frac{\D}{k!} n_F(T) = \frac{\e}{\e'} n_F(T).$$
	Let $S_1 \sub Y_1,\ldots,S_k \sub Y_k$ be subsets with $n_F(G[S_1,\ldots,S_k]) \ge \e' \cdot n_F(G)$. 	
	Then, by the above inequality, $n_F(G[S_1,\ldots,S_k]) \ge \e \cdot n_F(T)$.
	Since $T$ is $\e$-$(H,F)$-regular, we thus have
	$$\coeff{H}{F}(G[S_1,\ldots,S_k]) = \coeff{H}{F}(T) \pm \e .$$
	A special case of the above, obtained by taking $S_i=Y_i$, is
	$$\coeff{H}{F}(G) = \coeff{H}{F}(T) \pm \e .$$
	Combining the last two estimates, we get
	$\coeff{H}{F}(G[S_1,\ldots,S_k]) = \coeff{H}{F}(G) \pm 2\e$.
	Since $2\e \le \e'$ we deduce that $G$ is $\e'$-$(H,F)$-regular, as needed.
\end{proof}

We will also need the following standard fact. 
\begin{fact}\label{fact:trivial-reg}
	If $G$ is a bipartite graph of density at most $\e^3$ then $G$ is $\e$-regular.
\end{fact}
To see why Fact~\ref{fact:trivial-reg} is true, note that all 
subsets $S$ and $T$, each of size at least an $\e$-fraction of its
vertex class, satisfy $d(S,T) \le d(G)/\e^2 \le \e$.

\subsection{Putting everything together}

We are now ready to prove Theorem~\ref{theo:main}.
Recall that a vertex partition $\P$ of a graph $G$ on $n$ vertices is said to be \emph{$\e$-regular} if $\sum_{X,X'} |X||X'| \le \e n^2$ where the sum is over all cluster pairs $X,X' \in \P$ such that $G[X,X']$ is not $\e$-regular.
We use the following standard terminology:
a (semi-) blowup is said to be \emph{balanced} if all vertex classes have the same size.
The \emph{common refinement} of partitions $\P$ and $\V$ is the partition $\{X \cap V \,\vert\, X \in \P, V \in \V, X \cap V \neq \emptyset\}$.

Our main result towards the proof of Theorem~\ref{theo:main} is the following, which in some sense lifts Claim~\ref{claim:reg-partite} to a statement about partitions. Note that in the construction here we need the blowup to be balanced, or at least not too ``unbalanced'', since a relatively small vertex class could be completely covered by small fragments of clusters from the regular partition, thereby making our slicing lemma in Claim~\ref{claim:slicing} unusable.
\begin{theo}\label{theo:main2}
	If $\P$ is an $\e$-$(H,F)$-regular partition of a balanced semi-blowup $H \sbu_e G_0$ then the common refinement of $\P$ and $\{V_1,V_2\}$ is an $\e^{1/4} k^{2k}$-regular partition of $G_0$.
\end{theo}
\begin{proof}
	Put $\e'=\e^{1/4} k^{2k}$.
	Let $G= H \odot_e G_0$ be a semi-blowup of $H$ with $n=|V_1|=\cdots=|V_k|$, 
	and let $\P$ be an $\e$-$(H,F)$-regular partition of $G$.
	For $i \in \{1,2\}$, denote the common refinement of $\P$ and $\{V_i\}$ by $\Q_i = \{ X \cap V_i \,\vert\, X \in \P,\, X \cap V_i \neq \emptyset\}$.
	We need to prove that
	\begin{equation}\label{eq:lb-goal}
	\sum_{\substack{(Y_1,Y_2) \in \Q_1\times\Q_2\\\text{not $\e'$-regular}}} |Y_1||Y_2| 
	\le \e' n^2 .
	\end{equation}
	Let $\P_i = \big\{ X \in \P \,\big\vert\, |X \cap V_i| \ge \d_i|X| \big\}$
	where we set, with hindsight,
	$$\d_i = 
	\begin{cases}
		\frac12\e'			&\text{if } i \in \{1,2\}\\
		\frac{1}{2k}	&\text{if } i \in \{3,\ldots,k\} \;.
	\end{cases}
	$$
	Moreover, let $\P_i^* = \big\{ X \cap V_i \,\big\vert\, X \in \P_i \big\}$.
	Using Claim~\ref{lemma:approx-cover} and the fact that $G$ is balanced,
	\begin{equation}\label{eq:approx-cover}
	\sum_{\substack{Y_3,\ldots,Y_k\colon\\\forall i \, Y_i \in \P_i^*}} \prod_{i=3}^k |Y_i|
	= \prod_{i=3}^k \sum_{Y_i \in \P_i^*} |Y_i|
	\ge \Big(\frac12n\Big)^{k-2} \;.
	\end{equation}
	%
	On the way to proving~(\ref{eq:lb-goal}), our first goal will be to prove an analogous bound for $\P_1^*\times\P_2^*$, rather than for $\Q_1\times\Q_2$, as follows;
	\begin{equation}\label{eq:lb-goal2}
	\sum_{\substack{(Y_1,Y_2) \in \P^*_1\times\P^*_2\\\text{not $\e'$-regular}}} |G_0[Y_1,Y_2]| 
	\le \e 2^{k} \cdot |G_0| \;.
	\end{equation}
	Put $\D=\d_1\cdots\d_k$ $(\ge {\e'}^2/(2k)^k)$.
	We claim that if the $k$-partite graph $G[X_1,\ldots,X_k]$ with $X_1\in \P_1,\ldots,X_k \in \P_k$ is $\e$-$(H,F)$-regular then the bipartite graph $G[X_1 \cap V_1, X_2 \cap V_2]$ is 
	$\e'$-regular.
	To see this, put $Y_i = X_i \cap V_i$ and note that $G[Y_1,\ldots,Y_k]$ is of the form $H \odot_e G[Y_1,Y_2]$, by Observation~\ref{obs:sbu-induced}, and since $|Y_i| \ge \d_i|X_i|$ for every $1 \le i \le k$, it is $\e k!\D^{-1}$-$(H,F)$-regular by Claim~\ref{claim:slicing}. 
	Next, Claim~\ref{claim:reg-partite} implies that $G[Y_1,Y_2]$ is 
	$\sqrt{\e k!\D^{-1}} k^{2k}$-regular. 
	Since $\sqrt{\e k!\D^{-1}} k^{2k} \le (\sqrt{\e}/\e') k^{3k} \le \e'$, our claim follows.
	Now, to simplify the discussion, let us introduce two pieces of notation:  
	$$\R_\P = \big\{(X_1,\ldots,X_k) \in \P_1\times\cdots\times\P_k \,\big\vert\, G[X_1,\ldots,X_k] 
	\text{ is not $\e$-$(H,F)$-regular} \big\}\,,$$ 
	$$\R_\P^* = \big\{(X_1,\ldots,X_k) \in \P_1\times\cdots\times\P_k \,\big\vert\, G[X_1 \cap V_1, X_2 \cap V_2] 
	\text{ is not $\e'$-regular} \big\} .$$
	Then the claim above, in contrapositive, says $\R_\P^* \sub \R_\P$. 
	We deduce that
	\begin{align*}
	\sum_{(X_1,\ldots,X_k) \in \R_\P} n_H(G[X_1,\ldots,X_k]) 
	&\ge \sum_{(X_1,\ldots,X_k) \in \R_\P^*} n_H(G[X_1,\ldots,X_k]) \\
	&\ge \sum_{(X_1,\ldots,X_k) \in \R_\P^*} n_H(G[X_1 \cap V_1,\ldots,X_k \cap V_k]) \\
	&= \sum_{\substack{Y_3,\ldots,Y_k\colon\\\forall i \, Y_i \in \P_i^*}} \sum_{\substack{(Y_1,Y_2) \in \P^*_1\times\P^*_2\\\text{not $\e'$-regular}}} n_H(G[Y_1,\ldots,Y_k])\\
	&= \sum_{\substack{Y_3,\ldots,Y_k\colon\\\forall i \, Y_i \in \P_i^*}} \sum_{\substack{(Y_1,Y_2) \in \P^*_1\times\P^*_2\\\text{not $\e'$-regular}}} |G[Y_1,Y_2]|\cdot|Y_3|\cdots|Y_k|\\
	&\ge \Big(\frac12 n\Big)^{k-2} \sum_{\substack{(Y_1,Y_2) \in \P^*_1\times\P^*_2\\\text{not $\e'$-regular}}} |G[Y_1,Y_2]| \;,
	\end{align*}
	where the last equality uses Observation~\ref{obs:sbu-induced} and Observation~\ref{obs:sbu-copies}, and the last inequality uses~(\ref{eq:approx-cover}).
	On the other hand, as $\P$ is an $\e$-$(H,F)$-regular partition of $G$, 
	$$\sum_{(X_1,\ldots,X_k) \in \R_\P} n_H(G[X_1,\ldots,X_k]) 
	\le \e \cdot n_H(G) = \e \cdot |G_0|n^{k-2} ,$$
	where the last equality uses Observation~\ref{obs:sbu-copies}.
	Combining the above two inequalities implies~(\ref{eq:lb-goal2}).
	
	To deduce from~(\ref{eq:lb-goal2}) that $\Q_1\cup\Q_2$ is indeed an $\e'$-regular partition of $G_0$, we will need two more observations.
	First, using Fact~\ref{fact:trivial-reg}, 
	$$\sum_{\substack{(Y_1,Y_2) \in \P^*_1\times\P^*_2\\\text{not $\e'$-regular}}} |Y_1||Y_2| \le \sum_{\substack{(Y_1,Y_2) \in \P^*_1\times\P^*_2\\\text{not $\e'$-regular}}} {\e'}^{-3}|G_0[Y_1,Y_2]| \le {\e'}^{-3}\e 2^k|G_0| \le \frac12\e'|G_0| \le \frac12\e' n^2 ,$$
	where the second inequality follows from~(\ref{eq:lb-goal2}), and the penultimate inequality uses the choice of $\e'$.
	Second, we claim that every cluster pair $(Y_1,Y_2) \in (\Q_1\times\Q_2) \setminus (\P_1^*\times\P_2^*)$ satisfies $|Y_1||Y_2| \le \d_1|X_1||X_2|$, where $X_i \in \P$ is uniquely defined by $Y_i = X_i \cap V_i$.
	Indeed, without loss of generality $Y_1 \notin \P_1^*$, meaning $|Y_1| \le \d_1|X_1|$, so  $|Y_1||Y_2| \le \d_1|X_1||X_2|$.
	It follows that
	$$\sum_{\substack{(Y_1,Y_2) \in \Q_1\times\Q_2 \colon\\(Y_1,Y_2) \notin \P_1^*\times\P_2^*}} |Y_1||Y_2| 
	\le \sum_{X,X' \in \P} \d_1|X \cap V_1||X' \cap V_2|
	= \d_1 \Big(\sum_{X \in \P} |X \cap V_1|\Big)\Big(\sum_{X \in \P} |X \cap V_2|\Big)
	= \d_1 n^2 .$$
	Combining the last two inequalities, and using our choice of $\d_1=\frac12\e'$,
	we obtain~(\ref{eq:lb-goal}) and thus complete the proof.
\end{proof}

To deduce Theorem~\ref{theo:main} from Theorem~\ref{theo:main2}, we use the celebrated tower-type lower bound for 
Szemer{\'e}di's regularity lemma, originally proved 
by Gowers~\cite{Gowers97}. (See also~\cite{MoshkovitzSh16} 
for a simpler proof.)
For our application we need 
the bound to hold without assuming the partition is equitable.\footnote{In an equitable partition, all parts have the same size give or take $1$.}
As was shown by Fox, Lov{\'a}sz and Zhao (see Theorem~2.2 in~\cite{FoxLoZh17}, or Theorem~1.2 in~\cite{FoxLo17}), this can be assumed without significant loss.

\begin{theo}[\cite{FoxLoZh17,Gowers97}]\label{theo:RL-lb}
	There is a bipartite graph whose every $\e$-regular partition is of order at least $\twr(1/\poly(\e))$.
\end{theo}

\begin{proof}[Proof of Theorem~\ref{theo:main}]
	Let $G_0$ be the bipartite graph given by Theorem~\ref{theo:RL-lb}.
	Consider a balanced semi-blowup $G=H \odot_e G_0$ (where, implicitly, the blowup has sufficiently many vertices for the semi-blowup $G$ to be well defined). 
	Let $\P$ be an $\e$-$(H,F)$-regular partition of $G$,
	and let $\Q$ be the common refinement of $\P$ and $\{V_1,V_2\}$, where $V_1,V_2$ are the vertex classes of $G_0$.
	Apply Theorem~\ref{theo:main2} to deduce that $\Q$ is a $\poly(\e k^k)$-regular partition of $G_0$.
	By Theorem~\ref{theo:RL-lb}, the order of $\Q$ is at least $\twr(1/\poly(\e k^k))$, 
	and we are done.
\end{proof}

\section{Concluding Remarks}

We have shown that tower-type lower bounds hold for every generalized regularity lemma.
One may raise the objection that, since a generalized regularity lemma can be meaningful also for sparse graphs (that is, of density tending to $0$ with the number of vertices), it is not as satisfactory if our lower bound construction is dense.
Indeed, the construction in Theorem~\ref{theo:main}, 
being a semi-blowup of $H$, is---for any $H$ with at least two edges---of 
density at least $1/k$.
However, it turns out that one can easily get arbitrarily sparse graph with the same lower bound, provided $F$ is connected.  
For example, one may add any desired number of isolated vertices to the graph from Theorem~\ref{theo:main} so as to reduce the density as needed, yet still maintain the property that any $(H,F)$-regular partition is large. Indeed, this is a consequence of the following observation.

\begin{observation}
	Suppose $F$ is connected.
	Let a $k$-partite graph $G$ be obtained from a $k$-partite graph $G_0$ by adding isolated vertices.
	If $G$ is $\e$-$(H,F)$-regular then $G_0$ is also $\e$-$(H,F)$-regular.
\end{observation}
\begin{proof}
	As $F$ and $H$ are connected, $n_F(G)=n_F(G_0)$ and $n_H(G)=n_H(G_0)$.
	Let $G'$ be a $k$-partite induced subgraph of $G_0$ with $n_F(G') \ge \e n_F(G_0)$.
	Then $n_F(G') \ge \e n_F(G)$, and as $G$ is $\e$-$(H,F)$-regular,
	$\coeff{H}{F}(G')=\coeff{H}{F}(G) \pm \e$,
	that is, 
	$\coeff{H}{F}(G')=\coeff{H}{F}(G_0) \pm \e$, as needed.
\end{proof}

\end{document}